\documentclass[a4paper, 12pt]{amsart}
\usepackage{amsmath,amssymb,amsthm}
\usepackage{amscd}
\usepackage{array,enumerate}
\newtheorem{Thm}{Theorem}[section]
\newtheorem{Prop}[Thm]{Proposition}
\newtheorem{Lem}[Thm]{Lemma}
\newtheorem{Cor}[Thm]{Corollary}
\newtheorem{Thmint}{Theorem}[section]

\theoremstyle{definition}
\newtheorem{Rem}[Thm]{Remark}

\newtheorem{Def}[Thm]{Definition}

\newcommand{\Cs}{C$^\ast$}

\newcommand{\id}{\mbox{\rm id}}

\newcommand{\rc}{\mathop{\rtimes _{\mathrm r}}}
\newcommand{\votimes}{\mathop{\bar{\otimes}}}

\newcommand{\bigvotimes}{\mathop{\bar{\bigotimes}}}
\newcommand{\vnc}{\mathop{\bar{\rtimes}}}

\newcommand{\wrr}[1]{\mathop{\wr _{#1}}}
\newcommand{\vwrr}[1]{\mathop{\bar{\wr}_{#1}}}

\newcommand{\vwr}{\mathop{\bar{\wr}}}

\DeclareMathOperator{\Wso}{{\rm W}^\ast}

\DeclareMathOperator{\stcl}{strong-closure}

\DeclareMathOperator{\bigfp}{\lower0.25ex\hbox{\LARGE $\ast$}}

\title[Rigid sides of simple AFD algebras]{Rigid sides of approximately finite dimensional simple operator algebras in non-separable category}

\author{Yuhei Suzuki}
\address{Graduate school of mathematics, Nagoya University, Chikusaku, Nagoya, 464-8602, Japan}
\email{yuhei.suzuki@math.nagoya-u.ac.jp}
\begin{document}
\begin{abstract}
Applying Popa's orthogonality method to a new class of groups,
we construct amenable group factors which are prime
and have no infinite dimensional regular abelian $\ast$-subalgebras.
By adjusting Farah--Katsura's solution of Dixmier's problem to the von Neumann
algebra setting, we obtain the first examples of prime
AFD factors and tensorially prime simple AF-algebras.
Our results are proved in ZFC, thus in particular
answering questions asked by Farah--Hathaway--Katsura--Tikuisis.
We also directly determine central sequences of certain crossed products.
This concludes the failure of the Kirchberg $\mathcal{O}_\infty$-absorption
theorem in the non-separable setting.
\end{abstract}
\maketitle
\section{Introduction}\label{section:intro}
The fundamental interest of operator algebra theory is on separable operator algebras.
Nevertheless, non-separable operator algebras
naturally arise even if one is only interested in countable objects.
For instance, they arise as ultraproducts, central sequence algebras, injective envelopes \cite{Ham79}, \cite{Ham85}, (generalized) Calkin algebras, uniform Roe algebras \cite{Roe93}, and so on.
Also, some set theorists and logicians care about problems on non-separable operator algebras as they often involve statements independent of ZFC and thus provide interesting sources to study axioms beyond ZFC (see e.g., \cite{Far}, \cite{Wea}).

The purpose of the present paper is to clarify
that even in the tamest class of simple operator algebras, that is, the approximately finite dimensional case,
non-separable operator algebras can be surprisingly rigid,
contrary to softness of simple separable amenable operator algebras (see e.g., \cite{Con}, \cite{Win}).

Throughout the paper, we say that a von Neumann algebra is {\it AFD}
(approximately finite dimensional)
if it admits an increasing net of finite dimensional $\ast$-subalgebras
whose union is strongly dense.
This corresponds to Murray and von Neumann's approximate finiteness (C)
introduced in \cite{MvN}.
Similarly, we say a \Cs-algebra is {\it AF} (approximately finite dimensional)
if there is an increasing net of finite dimensional $\ast$-subalgebras
with the norm dense union.
In this paper, we exclude finite dimensional algebras from these classes.

In \cite{FHKT}, Farah--Hathaway--Katsura--Tikuisis studied non-separable approximately finite dimensional operator algebras.
{\it Under the continuum hypothesis}, they have constructed
an AFD type II$_1$ factor and a simple unital AF-algebra
whose central sequences are all trivial.
This concludes that even for approximately finite dimensional algebras,
in the non-separable category, the tensor product structure
of approximately finite dimensional operator algebras is not necessary flexible unlike the separable case.
(We recall that in the separable case, any AFD factor tensorially absorbs the AFD type II$_1$ factor,
and any simple AF-algebra tensorially absorbs the Jiang--Su algebra $\mathcal{Z}$ \cite{JS}.)
The problem if these results are provable in ZFC was left open 
(\cite{FHKT}, Question 4.3).
In this paper, we settle these problems.
In fact, we solve their questions in the following stronger form.

\begin{Thmint}[Corollaries \ref{Cor:primeAFD} and \ref{Cor:primeAF}]\label{Thmint:prime}
The following statements hold true in ZFC.
\begin{enumerate}\upshape
\item There is a prime AFD type II$_1$ factor without infinite dimensional regular abelian $\ast$-subalgebras.
\item There is a tensorially prime simple unital AF-algebra 
without infinite dimensional regular abelian $\ast$-subalgebras.
\end{enumerate}
\end{Thmint}
We notice that the obstruction on regular abelian $\ast$-subalgebras
yields that these operator algebras cannot be realized
as a (twisted) groupoid algebra of an \'{e}tale groupoid $G$ with $G^{(0)}$ an infinite set.
See Section \ref{subsection:regular} and references given there for more details.
We also note that R{\o}rdam \cite{Ror} has constructed examples of simple unital nuclear
tensorially prime \Cs-algebras.
More precisely, he has constructed simple unital (separable) nuclear \Cs-algebras
which are neither stably finite nor purely infinite.
Tensorial primeness of such \Cs-algebras then follows from Kirchberg's dichotomy (see Theorem 4.1.10 of \cite{Ror02}).
To the author's knowledge, R{\o}rdam's examples
are the only previously known examples of tensorially prime simple nuclear \Cs-algebras.

To construct such algebras, we introduce a new class of groups
which we call {\it iterated wreath product groups} (along an ordinal).
We then apply Popa's orthogonality method \cite{Pop83} to von Neumann algebras of iterated wreath product groups to confirm the primeness and the absence of regular abelian $\ast$-subalgebras.
Then, to obtain the AFD condition of these group von Neumann algebras,
we extend Murray and von Neumann's theorem on the equivalence of
three approximate finiteness conditions for separable type II$_1$ factors
to type II$_1$ factors of density character $\aleph_1$ (Proposition \ref{Prop:AFD}). This is a von Neumann algebra analogue of Proposition 5.6 in \cite{FK}. By combining this result and Connes's classification theorem \cite{Con}, we obtain the desired AFD type II$_1$ factor and AF-algebra.

We also directly determine central sequences for certain crossed products of iterated wreath product groups.
This allows to answer Question 4.3 of \cite{FHKT} even in the class of purely infinite algebras.
In particular we conclude the failure of the Kirchberg $\mathcal{O}_\infty$-absorption theorem (\cite{KP}, Theorem 3.15) in the non-separable case.
Here we recall that the Kirchberg $\mathcal{O}_\infty$-absorption theorem
states that any simple separable nuclear purely infinite \Cs-algebra (called a Kirchberg algebra)
tensorially absorbs the Cuntz algebra $\mathcal{O}_\infty$.
Flexibility of the tensor product structure of amenable operator algebras, e.g., $\mathcal{O}_\infty$-stability, $\mathcal{Z}$-stability, the McDuff property, play a central role in the classification theory of operator algebras (see e.g., \cite{Con}, \cite{Win}).
\begin{Thmint}[Corollary \ref{Cor:NSKT}, Remark \ref{Rem:Cartan}]\label{Thmint:central} 
The following statements hold true in ZFC.
\begin{enumerate}\upshape
\item There is a simple unital AF-algebra
with no non-trivial central sequences.
\item There is an AFD type II$_1$ factor
with no non-trivial central sequences.
\item There is an AFD type III factor with no non-trivial central sequences.\item There is a simple unital nuclear purely infinite \Cs-algebra
with no non-trivial central sequences.
\end{enumerate}
Moreover, one can construct these algebras as the reduced crossed product
of a group action on an abelian operator algebra.
\end{Thmint}
By the last statement, algebras in Theorem \ref{Thmint:central} can be non-isomorphic
to algebras in Theorem \ref{Thmint:prime}.
Furthermore, finite dimensional approximations
and AF-algebras in Theorem \ref{Thmint:central}
are directly given. In particular we avoid Connes's classification theorem \cite{Con}
and Proposition \ref{Prop:AFD}.
We remark that a unital \Cs-algebra with no non-trivial central sequences
cannot have non-trivial separable tensor product components.
See Section \ref{subsection:central} for details.

We also construct tensorially prime simple unital \Cs-algebras of decomposition rank one. Here we do not review the decomposition rank of \Cs-algebras. We refer the reader to \cite{KW} for the definition of decomposition rank
and \cite{Win} for its important role in the classification theory of \Cs-algebras.
\begin{Thmint}[Theorem \ref{Thm:dr1}]\label{Thmint:dr1}
There is a tesorially prime simple unital \Cs-algebra of decomposition rank one with no non-trivial central sequences.
\end{Thmint}
We note that AF-algebras are of decomposition rank zero (see Examples 6.1 (i) in \cite{KW}).
Thus these (nuclear) \Cs-algebras have a different feature to those in Theorems \ref{Thmint:prime} and \ref{Thmint:central}.
To show tensorial primeness, we again apply Popa's orthogonality method.
\subsection*{Organization of the paper}
In Section \ref{section:pre}, we summarize basic facts and terminologies.
Besides from known facts, we include a new general statement on
finite GNS-completions of regular abelian $\ast$-subalgebras (Proposition \ref{Prop:regular}).
This in particular strengthens and unifies Corollaries 4.7 and 4.11 of \cite{LR} (see Corollary \ref{Cor:ssolid}). In Section \ref{section:iterated}, we introduce and study iterated wreath product groups. In Section \ref{section:AFD}, we establish the equivalence
of approximate finiteness (A), (B), (C) of Murray--von Neumann for type II$_1$ factors of density character $\aleph_1$. In Section \ref{section:prime},
we prove Theorem \ref{Thmint:prime}. In Section \ref{section:central},
we study central sequences of certein crossed products.
As a consequence, we obtain Theorem \ref{Thmint:central}.
In Section \ref{section:dr1}, by modifying the constructions and proofs in
Sections \ref{section:prime} and \ref{section:central}, we prove Theorem \ref{Thmint:dr1}.

\section{Preliminries}\label{section:pre}
Here we recall a few basic facts.
Our basic references on operator algebras are \cite{BO}, \cite{KR}.
We first fix notations which we frequently use.
\subsection*{Notations}
\begin{itemize}
\item  The symbols `$\votimes$' and `$\otimes$' stand for the von Neumann algebra tensor product and the minimal tensor product respectively.
\item The symbols `$\vnc$' and `$\rc$' stand for the von Neumann algebra crossed product and the reduced \Cs-crossed product respectively.
\item For a subset $S$ of a von Neumann algebra $M$,
denote by $\Wso(S)$ the von Neumann subalgebra of $M$ generated by $S$.
\item For a von Neumann algebra $M$, $U(M)$ denotes the unitary group of $M$.
\item We denote the unit element of a group by $e$.
(The ambient group is always clear from the context.) 
\end{itemize}
\subsection{Ordinals}
For the definition and basic facts on (von Neumann) ordinals, see Chapter 1 of \cite{Kun} for instance.
Here we only recall a few basic terminologies which we frequently use.
As usual, denote by $\omega_1$ the smallest uncountable ordinal.
Denote by $\aleph_0$ and $\aleph_1(= \omega_1)$ the countable infinite cardinal and the smallest uncountable cardinal respectively.
For each ordinal $\alpha$, there is at most one ordinal $\beta$
with $\alpha= \beta+1$.
If such $\beta$ exists, $\alpha$ is said to be a successor ordinal,
and $\beta$ is said to be the predecessor of $\alpha$.
Otherwise $\alpha$ is said to be a limit ordinal.
A limit ordinal $\alpha$ is said to be of {\it uncountable cofinality}
if there is no countable subset $S$ in $\{ \beta : \beta < \alpha\}$
satisfying $\sup(S) =\alpha$.
For instance, $\omega_1$
is of uncountable cofinality, as the union of countably many countable sets is again countable.
\subsection{Tensorial primeness}
A von Neumann algebra $M$ is
said to be {\it prime} if for any isomorphism $M \cong M_1 \votimes M_2$,
one of $M_i$ is type $I$.
Similarly, a \Cs-algebra $A$ is said to be {\it tensorially prime}
if for any isomorphism $A \cong A_1 \otimes A_2$,
one of $A_i$ is a type $I$ \Cs-algebra.
(Since the term `primeness' is reserved for another property of \Cs-algebras,
we use this terminology.)
It is clear that any strongly dense \Cs-subalgebra of a prime type II$_1$ factor is tensorially prime.

\subsection{Normalizers and regularity for subalgebras}\label{subsection:regular}
Throughout the paper, we only consider nonzero subalgebras.
For a $\ast$-subalgebra $N\subset M$ of a von Neumann algebra $M$,
define the {\it normalizer} of $N$ in $M$ to be
\[\mathcal{N}_M(N):= \{ u\in U(M): u N u^\ast =N\}.\]
A $\ast$-subalgebra $N$ of $M$ is said to be {\it regular}
if it satisfies $\mathcal{N}_M(N)''=M$.
A typical regular abelian $\ast$-subalgebra arises from the crossed product construction.
Indeed, if $\Gamma$ is a group acting on an abelian von Neumann algebra $A$,
then the inclusion $A \subset A \vnc \Gamma$ is regular.

For a $\ast$-subalgebra $ A \subset B$ of a \Cs-algebra $B$,
define the {\it normalizer} of $A$ in $B$ to be
\[N_B(A) := \{ n\in B: n A n^\ast \cup n^\ast A n \subset A\}.\]
A \Cs-subalgebra $A$ of $B$ is said to be {\it regular} if $N_B(A)$ generates
$B$ as a \Cs-algebra.
We note that, in this case, $N_B(A)$ spans a dense subspace of $B$.
Regular abelian $\ast$-subalgebras naturally arise from the groupoid constructions.
More precisely, for an \'etale locally compact Hausdorff groupoid $G$ with a twist $\Sigma$,
the inclusions $C_0(G^{(0)}) \subset {\rm C}^\ast_{\rm r}(G, \Sigma)$ and
 $C_0(G^{(0)}) \subset {\rm C}^\ast(G, \Sigma)$ are regular (for their definitions, see Section 4 in \cite{Ren08}).
For the proof of this fact, see Corollary 4.9 in \cite{Ren08}.
(Note that the second countability assumption is not used for this conclusion, and the same proof works for
the full twisted groupoid \Cs-algebras.)
Thus the absence of regular abelian $\ast$-subalgebras
implies the absence of (twisted) groupoid \Cs-algebra realizations.
Obviously, the definitions of regularity as a \Cs- and von Neumann algebra
are very different. For subalgebras of von Neumann algebras, we always consider regularity in the von Neumann algebra sense.

We next record a few observations on regular abelian $\ast$-subalgebras.
The next lemma is standard and well-known.
Since the author does not know a reference for this result, we include a proof.
\begin{Lem}\label{Lem:diffusevn}
Let $A \subset M$ be a regular abelian $\ast$-subalgebra of a type II$_1$ factor $M$.
Then the strong closure of $A$ contains $1_M$ and is either diffuse or finite dimensional.
\end{Lem}
\begin{proof}
By taking the strong closure if necessary, we may assume that $A$ is strongly closed in $M$.
We first show that $1_M \in A$.
Let $1_A$ denote the unit of $A$.
Then for any $u \in \mathcal{N}_M(A)$, $1_A= u1_A u^\ast$.
By the factoriality of $\mathcal{N}_M(A)''$($=M$), we have $1_A=1_M$.

Now assume we have a minimal projection $p$ in $A$.
Then for any $u \in \mathcal{N}_M(A)$,
$u p u^\ast$ is again a minimal projection in $A$.
Since $A$ is abelian, for any $u, v \in \mathcal{N}_M(A)$,
we have one of the relations $upu^\ast = vpv^\ast$ or $upu^\ast vpv^\ast=0$.
Then, since $M$ is finite,
the set
\[P_p:=\{ u p u^\ast : u \in \mathcal{N}_M(A)\}\]
is finite.
Define $r := \sum_{q \in P_p} q$.
Then for any $u\in\mathcal{N}_M(A)$, we have $uru^\ast =r$.
Hence $r \in A \cap \mathcal{N}_M(A)' = \mathbb{C}$.
Therefore $A$ is finite dimensional.
\end{proof}

We next prove the following general proposition.
This implies a \Cs-algebra analogue of Lemma \ref{Lem:diffusevn}.
We note that the last statement is shown in Lemma 4.10 of \cite{LR} (see also Corollary 4.11 in \cite{LR}) for Cantan subalgebras in separable \Cs-algebras. We point out that our proof does not rely on the groupoid realization theorem \cite{Ren08} unlike their proof.
\begin{Prop}\label{Prop:regular}
Let $C\subset A$ be a regular abelian $\ast$-subalgebra of a \Cs-algebra $A$.
Let $\pi \colon A \rightarrow \mathbb{B}(H)$
be a $\ast$-representation such that $\pi(A)''$ is finite.
Then $\pi(C)''$ is regular in $\pi(A)''$.
In particular, if $\pi(A)''$ is a type II$_1$ factor,
then $\pi(C)''$ is either diffuse or finite dimensional.
\end{Prop}
\begin{proof}
Note that the last statement follows from
the first statement and Lemma \ref{Lem:diffusevn}.

Put $\mathcal{C}:= \stcl(\pi(C))$, $\mathcal{A}:= \pi(A)''$.
We first show that the unit $p$ of $\mathcal{C}$ commutes with $\mathcal{A}$.
For any $n\in \pi( N_A(C))$, as $n \mathcal{C} n^\ast \cup n^\ast \mathcal{C} n \subset \mathcal{C}$, direct computations imply
\[(pn -pnp)^\ast(pn -pnp) =0,~ (np -pnp)(np -pnp)^\ast=0.\]
This yields $p\in \mathcal{A} \cap \pi( N_A(C))'= \mathcal{A} \cap \mathcal{A}'$.
Thanks to the centrality of $p$ in $\mathcal{A}$,
the regularity of the inclusion $\pi(C)'' \subset \mathcal{A}$
is equivalent to the regularity of $\mathcal{C} \subset \mathcal{A}p$.
Hence, by restricting $\pi$ to $pH$ if necessary, we may assume
$1_H \in \mathcal{C}$.

We show that, for any $n\in \pi(N_A(C))$,
there is $u\in \mathcal{N}_{\mathcal{A}}(\mathcal{C})$
satisfying $n=u|n|$. Since $|n| =(n^\ast 1_H n)^{1/2} \in \mathcal{C}$ for all $n \in \pi(N_A(C))$, this implies the regularity of $\mathcal{C} \subset \mathcal{A}$.
To show the claim, we first show that for any $n \in \pi(N_A(C))$,
denoting by $n= v|n|$ the polar decomposition of $n$,
we have $v \mathcal{C} v^\ast \cup v^\ast \mathcal{C} v\subset \mathcal{C}$.
To see this, for $k\in \mathbb{N}$, define a Borel function $f_k \colon \mathbb{R} \rightarrow \mathbb{R}$ to be $f_k(t) = t^{-1}$ for $t\in [1/k, k]$ and $f_k(t)=0$ otherwise. Set $a_k:=f_k(|n|)$.
Then, since $a_k \in \mathcal{C}$, we have
\[na_k \mathcal{C} a_k n^\ast \subset \mathcal{C}.\]
As the bounded sequence $(n a_k )_{k=1}^\infty$ $\ast$-strongly converges to $v$, we have $v \mathcal{C} v^\ast \subset \mathcal{C}$.
By applying the same argument to $n^\ast$, we obtain the other inclusion.

Define the set
\[\mathcal{V}:= \{ w\in \mathcal{A}: w{\rm~is~a~partial~isometry}, w^\ast \mathcal{C} w \cup w \mathcal{C} w^\ast \subset \mathcal{C}, wv^\ast v=v\}.\]
Note that $v\in \mathcal{V}$ hence $\mathcal{V}$ is not empty.
We equip the set $\mathcal{V}$ with the partial order $\preceq$ defined as follows.
For two elements $w_1, w_2 \in \mathcal{V}$,
we declare $w_1 \preceq w_2$ if $w_2 w_1^\ast w_1 =w_1$.
It is not hard to see that $(\mathcal{V}, \preceq)$ satisfies the assumption of Zorn's lemma.

Choose a maximal element $u$ of $\mathcal{V}$.
We show that $u$ is a unitary element.
To lead to a contradiction, assume that $u$ is not a unitary element.
Since $\mathcal{A}$ is finite,
the projections $p:=1-u^\ast u$ and $q:= 1-uu^\ast$ are Murray--von Neumann equivalent in $\mathcal{A}$.
Since $\pi(N_A(C))$ spans a strongly dense subspace of $\mathcal{A}$,
one can choose $m\in \pi(N_A(C))$ satisfying
$x:=qmp \neq 0$.
Since both $p, q$ are in $\mathcal{C}$,
we have $x \mathcal{C} x^\ast \cup x^\ast \mathcal{C} x \subset \mathcal{C}$.
Let $x= w|x|$ be the polar decomposition of $x$.
By the argument in the second previous paragraph, we have $w \mathcal{C} w^\ast \cup w^\ast \mathcal{C} w \subset \mathcal{C}$.
Also, since $x= q xp$, we have $w=qwp$.
This implies $w u^\ast = w^\ast u =0$.
Hence $s:=u+w$ is a partial isometry in $\mathcal{A}$.
It is clear from the definition that $s u^\ast u= u$
and $s\neq u$.
Moreover, we have $w \mathcal{C}u^\ast = w u^\ast u \mathcal{C} u^\ast =0$
and similarly $w^\ast \mathcal{C} u= 0$.
Hence
\[s \mathcal{C} s^\ast \cup s^\ast \mathcal{C} s \subset \mathcal{C}.\]
This contradicts the maximality of $u$.
We thus conclude $u\in \mathcal{N}_{\mathcal{A}}(\mathcal{C})$.
Since $uv^\ast v= v$,
we obtain the desired presentation $n=u|n|$.
\end{proof}
As an immediate consequence of Proposition \ref{Prop:regular}, we obtain the following corollary.
This strengthens and unifies Corollaries 4.7 and 4.11 in \cite{LR} (cf.~Corollary B of \cite{OP2}, Theorem B of \cite{CS}).
We again emphasize that our proof does not involve groupoids.
We recall from \cite{OP}, \cite{OP2} that
a von Neumann algebra $M$ is said to be {\it strongly solid}
if for any diffuse amenable von Neumann subalgebra $A$ of $M$,
$\mathcal{N}_M(A)''$ is amenable.
\begin{Cor}\label{Cor:ssolid}
For any non-amenable strongly solid type II$_1$ factor,
all its strongly dense \Cs-subalgebras have no infinite dimensional regular abelian $\ast$-subalgebras.
\end{Cor}
\subsection{Popa's orthogonality method}
Recall that a von Neumann algebra $M$ is said to be {\it countably decomposable}
if there is no uncountable family consisting of pairwise orthogonal nonzero projections in $M$.
Throughout the paper, for a countably decomposable finite von Neumann algebra $M$, we fix a faithful normal tracial state $\tau$ (otherwise specified).
In most cases, $M$ is a type II$_1$ factor and thus the choice of $\tau$ is canonical.
Our analysis in Section \ref{section:prime} (thus Theorem \ref{Thmint:prime}) is based on {\it Popa's orthogonality method} \cite{Pop83}.
Here we briefly review this method and related terminologies.
In the seminal paper \cite{Pop83}, Popa has developed a powerful strategy to determine
the normalizers of operator subalgebras.
Among other things, his method provides the following striking results;
the first examples of prime type II$_1$ factors (e.g., uncountable free group factors),
the first examples of type II$_1$ factors with no infinite dimensional regular abelian $\ast$-subalgebras (e.g., uncountable free group factors, ultraproducts of type II$_1$ factors). For more applications, see the original article \cite{Pop83}.
His analysis of the normalizer is based on the orthogonality relation between
two von Neumann subalgebras.
Here recall from Definition 2.2 of \cite{Pop83} that
two $\ast$-subalgebras $A, B$ in $M$ are said to be {\it orthogonal},
denoted by $A \perp B$,
if they satisfy the following condition:
$\tau(a b) = \tau(a) \tau(b)$ for all $a\in A$ and $b\in B$.
This condition has a number of useful characterizations; see Lemma 2.1 in \cite{Pop83}.
\subsection{Density character}
The {\it density character} of a topological space $X$
is
the cardinal \[\min\{\sharp S: S \subset X {\rm ~dense}\}.\]
For von Neumann algebras, we always consider
the density character with respect to the $\sigma$-strong topology.
Note that the weak topology, strong topology, strong $\ast$-topology,
$\sigma$-weak topology, $\sigma$-strong $\ast$-topology
define the same density character as the $\sigma$-strong topology. This follows from
the proof of statement (1) below.
The following proposition clarifies aspects of the density character of von Neumann algebras.
For future reference, we record the proof.
\begin{Prop}\label{Prop:dc}
\begin{enumerate}[\upshape(1)]
\item The density character of $M$ is equal to
\[\min\{\sharp S:S\subset M {\rm~infinite~ subset}, \Wso(S)=M\}.\]
\item When $M$ is countably generated, the density character of $M$ coincides with that of the predual $M_\ast$ (with respect to the norm topology).
\item When $M$ is infinite dimensional, for any faithful normal state $\varphi$ on $M$, 
the density character of $M$ is equal to the dimension
of the GNS Hilbert space $L^2(M, \varphi)$.
\end{enumerate}
\end{Prop}
\begin{proof}[Proof $($outline$)$]
(1):
We observe that, for any infinite subset $S$ of $M$,
the $\ast$-subalgebra of $M$ generated by $S$ over the field $\mathbb{Q}(i)$
has the same cardinal as $S$.
This proves the claim.

(2): The duality $(M_\ast)^\ast =M$ shows that
the density character of $M$ is at most that of $M_\ast$.
To see the converse, take a faithful normal state $\varphi$ on $M$.
Then the Hahn--Banach theorem shows that
the set $\{\varphi(x \cdot): x\in M\}$ is dense in $M_\ast$.

(3): Since $M \subset \mathbb{B}(L^2(M, \varphi))$,
the density character of $M$ is at most the dimension of $L^2(M, \varphi)$.
The converse follows from the following observation.
For any strongly dense subset $S$ of $M$,
its canonical image in $L^2(M, \varphi)$ is dense.
\end{proof}

For \Cs-algebras, we consider the density character with respect to the norm topology.
Obviously the analogue of Proposition \ref{Prop:dc} (1) holds true in the \Cs-algebra context.
\subsection{Central sequences}\label{subsection:central}
For a unital \Cs-algebra $A$,
a bounded sequence $(x_n)_{n=1}^\infty$ of $A$ is said to be {\it central}
if it satisfies $\lim_{n \rightarrow \infty} \| x_ny -y x_n\| =0$ for all $y\in A$.
A central sequence $(x_n)_{n=1}^\infty$ in $A$ is said to be {\it trivial}
if there is a sequence $(\lambda_n)_{n=1}^\infty$ of scalars
satisfying $\lim_{n \rightarrow \infty} \| x_n -\lambda_n\|=0$.
We say that $A$ has {\it no non-trivial central sequences}
if all central sequences of $A$ are trivial.
We note that if a unital \Cs-algebra $A$ has no non-trivial central sequences, then $A$ has no separable tensor product factors other than full matrix algebras.
To see this, it suffices to show that any unital separable \Cs-algebra $B$
except full matrix algebras has a non-trivial central sequence.
The case $B$ is simple is shown in Proposition 2.10 of \cite{Kir}.
When $B$ is not simple, for a proper ideal $J$ of $B$,
a quasi-central approximate unit $(e_n)_{n=1}^\infty$ of $J$ in $B$
defines a non-trivial central sequence in $B$.

Similarly, for a von Neumann algebra $M$,
an (operator norm) bounded sequence $(x_n)_{n=1}^\infty$ of $M$
is said to be {\it central} if for any $y\in M$,
the sequence $(x_ny - y x_n)_{n=1}^\infty$ $\ast$-strongly converges to $0$.
A central sequence $(x_n)_{n=1}^\infty$ in $M$ is said to be {\it trivial}
if there is a sequence $(\lambda_n)_{n=1}^\infty$ of scalars
satisfying $x_n - \lambda_n \rightarrow 0$ $\ast$-strongly as $n$ tends to infinity.
We say that $M$ has {\it no non-trivial central sequences}
if all central sequences of $M$ are trivial.

We remark that, for von Neumann algebras, the definitions of centrality and triviality in the von Neumann and \Cs-algebra sense are very different.
In this paper, for von Neumann algebras, we always understand
these terminologies in the von Neumann algebra sense.
\section{Iterated wreath product groups and their basic properties}\label{section:iterated}

Throughout the paper, we assume that groups are discrete.
In this section, we introduce a new class of groups which we call {\it iterated wreath product groups}.
We also study their basic properties which we use.

We first recall the wreath product construction.
For two groups $\Gamma$ and $\Lambda$,
the wreath product group $\Gamma \wr \Lambda$
is given by
$(\bigoplus_{\Lambda} \Gamma) \rtimes_\sigma \Lambda$.
Here $\sigma \colon \Lambda \curvearrowright \bigoplus_{\Lambda} \Gamma$
denotes the left shift action.
We identify $\Lambda$ with the subgroup of $\Gamma \wr \Lambda$
via the canonical embedding $ \Lambda \ni g \mapsto (e, g) \in \Gamma \wr \Lambda$.

The main idea of the construction is to iterate
the wreath product construction along an ordinal.
More precisely, we construct a new group as follows.
Given an ordinal $\alpha$ and
a family $(\Gamma_\beta)_{\beta < \alpha}$ of groups.
We then define an increasing family $(G_\beta)_{\beta \leq \alpha}$
of groups as follows.
We first set $G_0 := \{ e \}$ (the trivial group).
Assume that the increasing family $(G_\gamma)_{\gamma<\beta}$
has been defined for an ordinal $\beta \leq \alpha$.
We then define $G_\beta$ as follows.
When $\beta$ is a limit ordinal,
we define $G_\beta := \bigcup_{\gamma < \beta} G_\gamma$.
When $\beta$ has the predecessor $\eta$,
we define
$G_\beta := \Gamma_\eta \wr G_\eta$.
This process extends the family
$(G_\gamma)_{\gamma<\beta}$ to $(G_\gamma)_{\gamma<\beta+1}$.
By the transfinite induction, we obtain the increasing family
$(G_\beta)_{\beta \leq \alpha}$ of groups.
We refer to the group $G:=G_\alpha$ as the {\it iterated wreath product of
$(\Gamma_\beta)_{\beta<\alpha}$}.
Throughout the paper, for $\beta<\alpha$, we denote by $G_\beta$ the subgroup of $G$ obtained in the construction of $G$ at the $\beta$-th step.

We next define the degree for elements in $G$.
For $g\in G$,
we define \[\deg(g):= \min \{ \beta \leq \alpha: g\in G_\beta\}.\]
By the definition of $(G_\beta)_{\beta \leq \alpha}$,
$\deg(g)$ must be a successor ordinal.

We now summarize basic properties of iterated wreath product groups.

The following lemma is useful to study properties of iterated wreath product groups.
Recall that a group is said to be {\it locally finite}
if it is a directed union of finite subgroups.
\begin{Lem}\label{Lem:lf}
Let $\alpha$ be an ordinal. Let $\mathcal{C}$ be a class of groups
closed under taking wreath products, isomorphisms, and directed unions
$($e.g., the class of amenable groups, the class of locally finite groups, the class of torsion-free groups$)$.
Let $(\Gamma_\beta)_{\beta < \alpha}$ be a family of groups in $\mathcal{C}$.
Then the iterated wreath product $G$ of $(\Gamma_\beta)_{\beta < \alpha}$ is contained in $\mathcal{C}$.
\end{Lem}
\begin{proof}
This follows by transfinite induction on $\alpha$.
\end{proof}
\begin{Lem}\label{Lem:ICC}
Let $\alpha$ be a limit ordinal.
Let $(\Gamma_\beta)_{\beta <\alpha}$ be a family of non-trivial groups.
Then the iterated wreath product $G$ of $(\Gamma_\beta)_{\beta <\alpha}$ is ICC.
\end{Lem}
\begin{proof}
Let $g\in G \setminus \{ e\}$ be given.
Put $\beta:=\deg(g)$.
Since $\alpha$ is a limit ordinal, for any $n\in \mathbb{N}$,
we have $\beta+n <\alpha$.
For each $n\in \mathbb{N}$, choose
$g_n\in \Gamma_{\beta + n}\setminus \{e\}$.
We identify $\Gamma_{\beta+n}$ with the $e$-th direct summand
of $\bigoplus_{G_{\beta+n}} \Gamma_{\beta+n} \subset G_{\beta+n+1}$.
Then $g_n g g_n^{-1} = g_n \sigma_g (g_n^{-1}) g \in G_{\beta+n +1}\setminus G_{\beta+n}$ for each $n\in \mathbb{N}$.
Hence the conjugacy class of $g$ in $G$ is infinite.
\end{proof}

The following disjointness property is the key of our results. Cf.~Proposition 4.1 in \cite{Pop83}.
\begin{Lem}\label{Lem:disjoint}
Let $\alpha$ be an ordinal.
Let $(\Gamma_\beta)_{\beta < \alpha}$ be a family of torsion-free groups.
Let $G$ be the iterated wreath product of $(\Gamma_\beta)_{\beta < \alpha}$.
Then for any $\beta <\alpha$ and for any $g\in G \setminus G_\beta$,
we have
\[G_\beta \cap g G_\beta g^{-1} = \{e \}.\]
\end{Lem}
\begin{proof}
Let $\gamma$ denote the predecessor of $\deg(g)$.
Note that by assumption, $\gamma \geq \beta$.
Since $G_\beta \subset G_\gamma$,
by replacing $\beta$ by $\gamma$ if necessary,
we may assume $g\in G_{\beta +1} \setminus G_\beta$.
Recall the definition $G_{\beta +1 } = \Gamma_\beta \wr G_\beta$.
By replacing $g$ by $gh$ for an appropriate $h \in G_\beta$,
we may assume
$g \in (\bigoplus_{G_\beta} \Gamma_\beta) \setminus \{ e \}$.
Then, since $G_\beta$ is torsion-free,
the (restricted) left shift action $\sigma \colon G_\beta \curvearrowright (\bigoplus_{G_\beta} \Gamma_\beta) \setminus \{ e \}$ is free.
Hence, for any $k\in G_\beta \setminus \{e \}$, we have
$g k g^{-1} = g \sigma_k(g^{-1}) k\in G \setminus G_\beta$.
\end{proof}
\section{Equivalence of Murray--von Neumann's approximate finiteness (A), (B), (C) for type II$_1$ factors of density character $\aleph_1$}\label{section:AFD}
In their fundamental work \cite{MvN}, Murray and von Neumann
developed the study of approximately finite dimensional von Neumann algebras.
They introduced three notions of approximate finiteness (A), (B), and (C), and show that all three conditions are equivalent for separable type II$_1$ factors.
In this section, following the strategy of Farah--Katsura \cite{FK},
we extend this result to type II$_1$ factors of density character $\aleph_1$.
Since the proof is basically the same as \cite{FK}, we only discuss the points where we need modifications.
Though in \cite{MvN} the definitions of approximate finiteness are stated in the separable case, we have obvious translations of these properties in general case as follows. 

Before stating them, we introduce the following notation.
Let $M$ be a type II$_1$ factor.
For $S, T \subset M$ and for $\delta>0$,
denote by $S\subset^\delta T$ if they satisfy the following condition:
for any $x\in S$, there is $y\in T$ with
$\|x - y\|_2 <\delta$.
Here $\| a \|_2:= \tau(a^\ast a)^{1/2}$ for $a \in M$.
\begin{Def}[Definitions 4.3.1, 4.5.2, and 4.6.1 in \cite{MvN}]
Let $M$ be a type II$_1$ factor.
\begin{enumerate}[(A)]
\item
We say that $M$ has {\it approximate finiteness (A)}
if it satisfies the following condition:
for any finite subset $F \subset M$ and any $\epsilon>0$,
there is an $n\in \mathbb{N}$ with the following property:
for any $q\geq n$, there is a type I$_q$ subfactor $N$ of $M$
with $ F \subset^\epsilon N$.
\item
We say that $M$ has {\it approximate finiteness (B)}
if it satisfies the following condition:
for any finite subset $F \subset M$ and any $\epsilon>0$,
there is a finite dimensional $\ast$-subalgebra $N$ of $M$
with $ F \subset^\epsilon N$.
\item 
We say that $M$ has {\it approximate finiteness (C)}
if it admits a strongly dense $\ast$-subalgebra obtained as a directed union of
 finite dimensional $\ast$-subalgebras.
\end{enumerate}
\end{Def}
In \cite{Ell}, it is shown that injectivity is equivalent to approximate finiteness (B) 
(certainly the proof depends on Connes's theorem \cite{Con}).
In a similar method to \cite{Ell}, one can show the equivalence of injectivity and approximate finiteness (A).
It is clear from the definition that (C) implies (B).
Hence to show the equivalence of (A), (B), (C), it suffices to show that injectivity implies (C).
Indeed, the following proposition holds true.
\begin{Prop}\label{Prop:AFD}
Let $M$ be an injective type II$_1$ factor of density character $\aleph_1$.
Then there is an increasing net $(M_i)_{i \in I}$ of
type I subfactors of $M$ whose union is strongly dense.
\end{Prop}
We point out that one can show the equivalence of three approximate finiteness conditions
for type II$_1$ factors of density character without using Connes's classification theorem.
See Remark \ref{Rem:Con} for details.
However we need Proposition 4.2 rather than just the equivalence of three approximate finiteness conditions.
We therefore give a direct proof of Proposition \ref{Prop:AFD}
(certainly with the aid of Connes's classification theorem \cite{Con}).

To prove Proposition \ref{Prop:AFD}, we need a few lemmas.
The next lemma plays the role of Lemma 5.1 in \cite{FK} in the von Neumann algebra setting. For a \Cs-algebra $A$,
we set $\mathcal{B}_A:= \{ x\in A: \| x \|\leq 1\}$.
\begin{Lem}\label{Lem:AFD}
For any $\epsilon>0$ and any $n\in \mathbb{N}$,
there exists $\delta>0$ with the following property.
For any type II$_1$ factor $M$,
for any type I$_n$ subfactor $D$ of $M$ and any type I$_{m}$ subfactor $B$ of $M$
satisfying $\mathcal{B}_D \subset ^\delta \mathcal{B}_B$ and $n|m$,
and for any subfactor $E\subset B \cap D$,
there is a unitary element $u \in M \cap E'$ satisfying
$u D u^\ast \subset B$, $\|u -1 \|_2 < \epsilon$.
\end{Lem}
\begin{proof}
Since $E$ is a type I subfactor contained in $B \cap D$, we have the tensor product decompositions
$M= E \votimes N$, $B= E \votimes P$, $D= E \votimes Q$,
where $N:= M \cap E'$, $P := B \cap E'$, $Q:= D \cap E'$.
Observe that $\mathcal{B}_D \subset ^\delta \mathcal{B}_B$
implies $\mathcal{B}_{P} \subset^\delta \mathcal{B}_{Q}$.
This reduces the proof of the statement to the case $E =\mathbb{C}$.
In this case, the statement is an immediate consequence of Lemmas 12.2.4 and 12.2.5 in \cite{KR}.
\end{proof}

We further introduce a few terminologies.
Let $\Lambda$ be a directed set. Let $M$ be a von Neumann algebra.
Recall from \cite{FK} that $\Lambda$ is said to be {\it $\sigma$-complete}
if any countable directed subset of $\Lambda$ has a supremum in $\Lambda$.
A family $(M_\lambda)_{\lambda \in \Lambda}$ of von Neumann subalgebras
of $M$
is said to be {\it directed} if
$\lambda \leq \mu$ if and only if $M_\lambda \subset M_\mu$.
Following \cite{FK}, we say that
a directed family $(M_\lambda)_{\lambda \in \Lambda}$
of von Neumann subalgebras of $M$
is {\it $\sigma$-complete} if it satisfies the following conditions:
\begin{itemize}
\item $\Lambda$ is $\sigma$-complete.
\item For any countable directed subset $S\subset \Lambda$,
we have $M_{\sup(S)} = \Wso(\bigcup_{s\in S}M_s)$.
\end{itemize}
We say that
a directed family $(M_\lambda)_{\lambda \in \Lambda}$
of von Neumann subalgebras of $M$ has {\it dense union}
if the union $\bigcup_{\lambda \in \Lambda} M_\lambda$ is strongly dense in $M$.

We next show the following basic lemma, which plays the role of Lemma 5.5 in \cite{FK}.
\begin{Lem}\label{Lem:directed}
Every type II$_1$ factor $M$ of density character $\aleph_1$ admits a directed $\sigma$-complete family of
separable subfactors with dense union indexed by $\omega_1$.
\end{Lem}
\begin{proof}
We first show that
any separable von Neumann subalgebra $N$ of $M$ is contained in a separable subfactor of $M$. To see this, take a countable strongly dense subset $S$ of $N$.
By Theorem 8.3.6 in \cite{KR},
for any $x\in M$, we have
\[\tau(x)\in \overline{\rm conv}\{ u x u^\ast: u \in U(M)\}.\]
Here and below, the symbol `$\overline{\rm conv}$'
stands for the norm closure of the convex hull.
Since $S$ is countable, one can choose a countable subset $T$ of $U(M)$
satisfying
\[\tau(x)\in \overline{\rm conv}\{ u x u^\ast: u \in T\} {\rm~for~ all~} x\in S.\]
Set $N_1 := \Wso(N \cup T)$.
Since $S$ is strongly dense in $N$,
it follows from the choice of $T$ that
any normal tracial state $\tau_1$ on $N_1$ satisfies
$\tau_1|_N = \tau|_N$.
By continuing this process, we obtain
an increasing sequence $(N_n)_{n=1}^\infty$
of separable von Neumann subalgebras in $M$
with the following property:
for any $n\in \mathbb{N}$ and for any normal tracial state $\tau_{n+1}$ on $N_{n+1}$,
we have $\tau_{n+1}|_{N_n}= \tau|_{N_n}$.
Now define $L:= \Wso(\bigcup_{n \in \mathbb{N}}N_n)$.
Then by the choice of $N_n$'s,
$L$ has a unique (faithful) normal tracial state.
Hence $L$ is a factor. Clearly $L$ is separable.

Now choose a strongly dense net $(x_\alpha)_{\alpha < \omega_1}$ in $M$.
We define a directed family $(M_\alpha)_{\alpha <\omega_1}$ of separable subfactors of $M$ as follows.
We first choose $M_0$ to be a separable subfactor of $M$ containing $x_0$.
For an ordinal $\alpha<\omega_1$, assume we have defined the family $(M_\beta)_{\beta<\alpha}$.
We then define $M_\alpha$ as follows.
When $\alpha$ is a limit ordinal,
set $M_\alpha := \Wso(\bigcup_{\beta<\alpha} M_\beta)$.
When $\alpha$ has the predecessor $\eta$,
fix $y\in M \setminus M_\eta$
and
define $M_\alpha$ to be a separable subfactor of $M$ containing $M_\eta \cup \{x_\gamma:\gamma\leq \alpha\} \cup \{y\}$.
By the transfinite induction,
we obtain a family $(M_\alpha)_{\alpha < \omega_1}$
of separable subfactors of $M$.
It is clear from the construction that the family is directed and $\sigma$-complete.
Moreover, by the construction,
$x_\alpha \in M_{\alpha+1}$ for all $\alpha<\omega_1$.
This shows the density of $(M_\alpha)_{\alpha < \omega_1}$ in $M$.
\end{proof}
\begin{proof}[Proof of Proposition \ref{Prop:AFD}]
By Connes's classification theorem \cite{Con},
each separable type II$_1$ subfactor $N$ of $M$
admits an increasing sequence of subfactors of type $I_{2^n}$, $n\in \mathbb{N}$, whose union is strongly dense in $N$.
The rest of the proof is the same as the proof of Proposition 5.6 in \cite{FK}
modulo the following modifications.
\begin{enumerate}
\item
Replace the operator norm by the norm $\| \cdot \|_2$.
\item Use Lemmas \ref{Lem:AFD}, \ref{Lem:directed} in place of Lemmas 5.1, 5.5 in \cite{FK} respectively.
\end{enumerate}
(For item (2), since our approximating sequence consists of full matrix algebras, Lemma \ref{Lem:AFD}, a slightly weaker analogue of Lemma 5.1 in \cite{FK},
is sufficient for our purpose.)
\end{proof}
\begin{Rem}\label{Rem:Con}
We point out that the equivalence of approximate finiteness (A) to (C) for type II$_1$ factors of density character $\aleph_1$ can be shown without using Connes's classification theorem (but using the classification theorem of separable AFD type II$_1$ factors due to Murray and von Neumann \cite{MvN}).
Indeed, the following version of Lemma \ref{Lem:directed} holds true:
every type II$_1$ factor $M$ of density character $\aleph_1$ with approximate finiteness (B) admits a directed $\sigma$-complete family of
separable AFD subfactors with dense union indexed by $\omega_1$.
To see this, it suffices to show that any separable von Neumann subalgebra $N$ of $M$ is contained in a separable AFD subfactor of $M$.
(The rest of the proof is similar to that of Proposition \ref{Prop:AFD}.)
Choose a strongly dense sequence $(x_{n, 1})_{n=1}^\infty$ of $N$.
By the assumption, one can choose a finite dimensional $\ast$-subalgebra $F_1$ of $M$ with $\{ x_{1, 1}\} \subset^{1} F_1$.
Set $N_1 := \Wso(N \cup F_1)$.
Choose a strongly dense sequence $(x_{n, 2})_{n=1}^\infty$ of $N_1$.
Again by the assumption, one can choose a finite dimensional $\ast$-subalgebra $F_2$ of $M$ with $\{ x_{i, j}: 1\leq i, j \leq 2 \} \subset^{1/ 2} F_2$.
Set $N_2:= \Wso(N_1 \cup F_2)$.
By iterating this construction, we obtain an increasing sequence $(N_n)_{n=1}^\infty$
of separable von Neumann subalgebras of $M$ and a strongly dense sequence $(x_{n, m+1})_{n=1}^\infty$ in $N_m$ for each $m\in \mathbb{N}$ with the following property:
for any $m\in \mathbb{N}$, there is a finite dimensional $\ast$-subalgebra $F$ of $N_{m}$ with $\{ x_{i, j} :1\leq i, j \leq m \} \subset ^{1/m} F$.
This implies approximate finiteness (B) of $L_1 := \Wso(\bigcup_{n=1}^\infty N_n)$.
We have seen in the proof of Lemma \ref{Lem:directed}
that there is a separable subfactor $P_1$ of $M$ containing $L_1$.
By iterating the constructions, we obtain an increasing sequence
\[N \subset L_1 \subset P_1 \subset L_2 \subset P_2 \subset \cdots \subset L_i \subset P_i \subset \cdots\]
of separable von Neumann subalgebras of $M$ such
that each $L_i$ satisfies approximate finiteness (B) and that each $P_i$
is a factor.
Now set $Q:= \Wso(\bigcup_{i=1}^\infty L_i)= \Wso(\bigcup_{i=1}^\infty P_i)$.
Obviously $N \subset Q$.
By definition, $Q$ is separable and satisfies approximate finiteness (B),
while by the second equality, $Q$ is a subfactor of $M$.
Now by Murray and von Neumann's theorem \cite{MvN},
$Q$ is a (separable) AFD subfactor of $M$.
\end{Rem}
\section{Constructions of tensorially prime approximately finite dimensional operator algebras}\label{section:prime}
In this section, under a mild condition, we show primeness
and absence of regular abelian $\ast$-subalgebras of
the group von Neumann algebras of (uncountable) iterated wreath product groups.
As an application, we prove Theorem \ref{Thmint:prime}.
\begin{Prop}\label{Prop:groupvn}
Let $\alpha$ be a limit ordinal of uncountable cofinality.
Let $(\Gamma_\beta)_{\beta <\alpha}$ be a family
of non-trivial torsion-free groups
and let $G$ denote its iterated wreath product.
Then the group von Neumann algebra $L(G)$ is prime and has
no infinite dimensional regular abelian $\ast$-subalgebras.
\end{Prop}
\begin{proof}
The proof is similar to Popa's proof of the same results for uncountable free group factors \cite{Pop83}.
We first observe that $L(G)$ is a type II$_1$ factor by Lemma \ref{Lem:ICC}.

Observe that
for any countable subset $S$ of $L(G)$, 
the set 
\[T:=\{g\in G: \tau(s u_g^\ast)\neq 0 {\rm~for~some~}s\in S\}\]
is countable.
Since $\alpha$ is of uncountable cofinality, we have
\[\beta:=\sup \left\{ \deg(g): g\in T \right\} < \alpha.\]
Note that $S \subset L(G_\beta)$.
This observation together with Proposition 4.1 in \cite{Pop83} and Lemma \ref{Lem:disjoint} shows
that for any diffuse, countably generated von Neumann subalgebra
$B$ of $L(G)$,
there is $\beta <\alpha$ with $\mathcal{N}_{L(G)}(B)'' \subset L(G_\beta)$.

We now show that $L(G)$ has no infinite dimensional regular abelian $\ast$-subalgebras.
To see this, assume we have an infinite dimensional regular abelian $\ast$-subalgebra
$A \subset L(G)$. By taking the strong closure, we may assume that $A$ is strongly closed.
Then, by Lemma \ref{Lem:diffusevn}, $A$ is diffuse and contains the unit of $L(G)$.
Choose a countably generated, diffuse
von Neumann subalgebra $A_0 \subset A$.
By the observation in the previous paragraph,
one can choose $\beta <\alpha$ satisfying $A \subset \mathcal{N}_{L(G)}(A_0)'' \subset L(G_\beta)$.
By applying Proposition 4.1 in \cite{Pop83} and Lemma \ref{Lem:disjoint} to $A$ itself,
we conclude $\mathcal{N}_{L(G)}(A)'' \subset L(G_\beta) \subsetneq L(G)$.
This contradicts the regularity of $A$ in $L(G)$.

We now prove the primeness of $L(G)$.
To lead to a contradiction, assume we have a tensor product decomposition
$L(G) =M_1 \votimes M_2$ where both $M_i$ are infinite dimensional.
Then as $L(G)$ is a type II$_1$ factor, so are both $M_i$.
Choose a countably generated, diffuse von Neumann subalgebra
$B_i$ of $M_i$ for $i=1, 2$.
By the argument in the second previous paragraph,
one can choose $\beta <\alpha$ satisfying $\mathcal{N}_{L(G)}(B_i)'' \subset L(G_\beta)$ for $i=1, 2$.
This implies the relations $M_{1} \subset \mathcal{N}_{L(G)}(B_2)'' \subset L(G_\beta)$ and $M_2 \subset \mathcal{N}_{L(G)}(B_1)''\subset L(G_\beta)$.
Since $G_\beta \subsetneq G$, this is a contradiction.
\end{proof}
\begin{Cor}\label{Cor:primeAFD}
There is a prime AFD type II$_1$ factor without infinite dimensional regular abelian $\ast$-subalgebras.
\end{Cor}

\begin{proof}
Let $(\Gamma_{\alpha})_{\alpha < \omega_1}$
be a family of non-trivial torsion-free countable amenable groups.
Let $G$ denote the iterated wreath product of $(\Gamma_{\alpha})_{\alpha < \omega_1}$. By Lemmas \ref{Lem:lf} and \ref{Lem:ICC}, $G$ is amenable and ICC.
Hence $L(G)$ is an injective type II$_1$ factor.
Since $\sharp G= \aleph_1$,
the density character of $L(G)$ is $\aleph_1$.
Propositions \ref{Prop:AFD} and \ref{Prop:groupvn}
now imply that $L(G)$ possesses the desired properties.
\end{proof}

A unital \Cs-algebra is said to be {\it monotracial} if there is a unique tracial state on it.
\begin{Cor}\label{Cor:primeAF}
There is a tensorially prime simple unital monotracial AF-algebra without infinite dimensional regular abelian $\ast$-subalgebras.
\end{Cor}
\begin{proof}
This is a consequence of Propositions \ref{Prop:regular}, \ref{Prop:AFD}, and Corollary \ref{Cor:primeAFD}.
\end{proof}
\begin{Rem}
\begin{enumerate}
\item By Proposition \ref{Prop:AFD}, \Cs-algebras constructed in Corollary \ref{Cor:primeAF} are in fact {\it AM-algebras} in the sense of \cite{FK}, \cite{FHKT}. That is, they are obtained as the inductive limits of full matrix algebras.
\item
By slightly modifying the proof of Proposition \ref{Prop:AFD}, for any UHF-algebra $A$,
one can arrange a \Cs-algebra obtained in Corollary \ref{Cor:primeAF}
to have the same ordered K-theory as $A$.
To see this, take a sequence $(k_n)_{n=1}^\infty$ of natural numbers
satisfying $A \cong \bigotimes_{n=1}^\infty \mathbb{M}_{k_n}(\mathbb{C})$.
Set $K_n := \prod_{i=1}^n k_i$ for each $n\in \mathbb{N}$.
Then, for any separable injective type II$_1$ factor,
by Connes's classification theorem \cite{Con},
it admits an increasing sequence of subfactors of type $I_{K_n}$, $n\in \mathbb{N}$,
with dense union.
Use these approximation sequences instead of the one consisting of type $I_{2^n}$ subfactors to proceed the proof of Proposition \ref{Prop:AFD}.
Then the resulting inductive system consists of full matrix algebras of size $K_n$, $n\in \mathbb{N}$.
Thus the inductive limit \Cs-algebra satisfies the required properties.
\item It is not hard to see from the proof of Proposition \ref{Prop:AFD} (cf.~\cite{FK})
that the AF-algebras we obtained in Corollary \ref{Cor:primeAF}
are of density character $\aleph_1$.

\end{enumerate}
\end{Rem}
We will see in the next section that the group von Neumann algebras
in Proposition \ref{Prop:groupvn} have no non-trivial central sequences (see Theorem \ref{Thm:nocentral}).
\section{Central sequences in certain crossed products}\label{section:central}
In this section, we study central sequences of the reduced crossed products
of non-commutative Bernoulli shifts of iterated wreath product groups.
We first introduce a few notations.
Let $A$ be a unital \Cs-algebra.
We set $\bigotimes_{\emptyset} A := \mathbb{C}$ for convenience.
For a group $G$ and for a $G$-set $S$ (possibly empty),
we equip $\bigotimes_S A$ with the
tensor shift $G$-action induced from the $G$-action on $S$.
We set \[A \wrr{S} G := \left(\bigotimes_S A \right) \rc G.\]
When $H$ is a subgroup of $G$ and
$T$ is a (possibly empty) $H$-invariant subset of $S$,
we identify
$A \wrr{T} H$
with a \Cs-subalgebra of $A\wrr{S} G$ in the canonical way.
Here we regard $T$ as an $H$-set by restricting the original action.
For a group $G$, we equip $G$ with the left translation $G$-action.
In the most important case $S=G$,
we simply denote $A \wrr{G} G$ by $A \wr G$ for short.
We use similar notations for von Neumann algebras.
For instance, for a von Neumann algebra $M$ with a faithful normal state $\varphi$
and for a group $G$, we set $(M, \varphi) \vwr G := \left[\bigvotimes_G (M, \varphi) \right]\vnc G$.

\begin{Thm}\label{Thm:nocentral}
Let $\alpha$ be a limit ordinal of uncountable cofinality.
Let $(\Gamma_\beta)_{\beta <\alpha}$ be a family of non-trivial groups
and let $G$ denote its iterated wreath product. Then the following statements hold true.
\begin{enumerate}\upshape
\item
Let $A$ be a unital \Cs-algebra.
Then $A \wr G$
has no non-trivial central sequences.
\item
Let $M$ be a von Neumann algebra with a faithful normal state $\varphi$.
Then $(M, \varphi)\vwr G$
has no non-trivial central sequences.
\end{enumerate}
\end{Thm}
\begin{proof}
We only prove (1).
The statement (2) can be shown by a similar argument.

Let $(x_n)_{n=1}^\infty$ be a central sequence in $A \wr G$.
Since $\alpha$ is of uncountable cofinality, one can find
$\beta < \alpha$ satisfying $\{ x_n : n \in \mathbb{N} \} \subset A \wr G_\beta$.
Fix a state $\varphi$ on $A$.
For any subgroup $H$ of $G$
and any left $H$-invariant subset $S$ of $G$,
by Exercise 4.1.4 in \cite{BO},
one can find a conditional expectation
$E_{H, S}^{\varphi} \colon A \wr G \rightarrow A \wrr{S} H$ satisfying
\[E_{H, S}^{\varphi} (x u_g) = \Phi_S(x)\chi_H(g)u_g {\rm~for~} x\in \bigotimes_G A, g\in G,\]
where $\Phi_S:= \left(\bigotimes_S \id_{A}\right) \otimes (\bigotimes_{G\setminus S}\varphi)$ and $\chi_H$ denotes the characteristic function of $H$ on $G$.
Observe that for two subgroups $H_1, H_2 \subset G$
and left $H_i$-invariant subsets $S_i \subset G$; $i=1, 2$,
we have
\[E_{H_1, S_1}^{\varphi} (A \wrr{S_2} H_2) = A \wrr{S_1 \cap S_2} (H_1 \cap H_2).\] 

Choose $g\in \Gamma_{\beta} \setminus \{e\}$.
We identify $\Gamma_{\beta}$ with the $e$-th direct summand
of $\bigoplus_{ G_\beta} \Gamma_{\beta} \subset G_{\beta +1}$.
Obviously $g G_\beta g^{-1} \cap G_\beta = \{ e\}$,
$g G_\beta \cap G_\beta = \emptyset$.
Since $u_g x_n u_g^\ast \in A \wrr{g G_\beta}(g G_\beta g^{-1})$,
we have $c_n:=E_{G_\beta, G_\beta}^{\varphi} (u_g x_n u_g^\ast) \in \mathbb{C}$.
Since $E_{G_\beta, G_\beta}^{\varphi} (x_n)=x_n$ for all $n$
and $\lim_{n \rightarrow \infty} \| x_n - u_g x_n u_g^\ast\| =0$,
we conclude
$\lim_{n \rightarrow \infty} \| x_n - c_n\| =0.$
\end{proof}
\begin{Rem}
The proof of Theorem \ref{Thm:nocentral} in fact shows the following slightly stronger statement:
for any iterated wreath product groups $G_1, \ldots, G_n$ as in Theorem \ref{Thm:nocentral}
and for any unital \Cs-algebras $A_1, \ldots, A_n$,
the tensor product
$\bigotimes_{i=1}^n ( A_i \wr G_i)$ has no non-trivial central sequences.
(The tensor products of conditional expectations used in the proof
of Theorem \ref{Thm:nocentral} play their role.)
A similar generalization is valid for von Neumann algebras.
\end{Rem}

We next give the following useful lemma on pure infiniteness of the reduced crossed product.
This is an immediate consequence of Lemma 3.2 of \cite{Kis}.
For the definition and consequences of pure infiniteness, we refer the reader to Chapter 4 of \cite{Ror02}.

Recall that an automorphism $\alpha$ of a \Cs-algebra $A$ is said to be {\it outer}
if there is no unitary element $u$ in the multiplier algebra $M(A)$
satisfying $\alpha(a)= uau^\ast$ for all $a\in A$.
An action $\alpha \colon \Gamma \curvearrowright A$ of a discrete group $\Gamma$ on a \Cs-algebra $A$ is said to be {\it outer} if
for any $s\in \Gamma \setminus \{e\}$, the automorphism $\alpha_s$ is outer.
\begin{Lem}\label{Lem:pi}
Let $\Gamma$ be a group.
Let $A$ be a simple purely infinite \Cs-algebra.
Let $\alpha \colon \Gamma \curvearrowright A$ be an outer action.
Then $A \rc \Gamma$ is simple and purely infinite.
\end{Lem}
\begin{proof}
It suffices to show the following claim.
For any $x = \sum_{g\in F} a_g u_g \in A \rc \Gamma$; $a_g\in A$, $e\in F \subset \Gamma$ finite subset,
with $a_e \geq 0$, $\| a_e\| >1$, for any positive element $b\in A$ of norm one, and for any $\epsilon>0$,
there is a contractive element $c\in A$ with
$\| c^\ast x c -b\| < \epsilon$.

By Lemma 3.2 of \cite{Kis},
one can choose a positive element $d_1 \in A$
satisfying $\|d_1 \| =1$,
$\| d_1 a_e d_1 \| \geq 1$,
and $ \sum _{g \in F \setminus \{e \}}\| d_1 a_g \alpha_g(d_1)\| < \epsilon/2$.
This implies
$\| d_1xd_1 - d_1 a_e d_1 \| <\epsilon/2$.
Since $A$ is simple and purely infinite,
there is $d_2\in A$
satisfying $\| d_2^\ast d_1 a_e d_1 d_2 - b \| <\epsilon/2$
and $\| d_2 \| \leq 1$.
Combining these inequalities,
we obtain
$\| d_2^\ast d_1 x d_1 d_2 -b \| < \epsilon$.
Thus $c:= d_1 d_2$ gives the desired element.
\end{proof}

The following corollary in particular shows that Kirchberg's $\mathcal{O}_\infty$-absorption theorem (\cite{KP}, Theorem 3.15)
fails in the non-separable setting.

\begin{Cor}\label{Cor:NSKT}
For any cardinal $\kappa > \aleph_0$,
there are operator algebras of density character $\kappa$ with no non-trivial central sequences in the following classes.
\begin{enumerate}\upshape
\item
Simple unital monotracial AF-algebras,
\item
AFD type II$_1$ factors,
\item
AFD type III factors,
\item
simple unital nuclear purely infinite \Cs-algebras.
\end{enumerate}
\end{Cor}
\begin{proof}
(1): 
Choose a family $( \Gamma_\alpha)_{\alpha< \omega_1}$
of locally finite non-trivial groups
satisfying $\sharp \Gamma_0= \kappa$ (e.g., $\Gamma_0 := \bigoplus_\kappa \mathbb{Z}_2$) and $\sharp \Gamma_\alpha \leq \aleph_0$ for $0<\alpha <\omega_1$.
Let $G$ be the iterated wreath product group of the family $(\Gamma_\alpha)_{\alpha< \omega_1}$.
By Lemma \ref{Lem:lf}, $G$ is locally finite.

Set $A:=\mathbb{M}_2(\mathbb{C}) \wr G$.
Clearly $A$ is unital.
By Theorem 3.1 of \cite{Kis}, $A$ is simple.
It follows from the construction of $G$ that $\sharp G= \kappa$.
Hence the density character of $A$ is $\kappa$.
Take an increasing net $(H_\lambda)_{\lambda \in \Lambda}$
of finite subgroups of $G$ with $\bigcup_{\lambda \in \Lambda} H_\lambda =G$.
Then the net $(\mathbb{M}_2(\mathbb{C}) \wr H_{\lambda})_{\lambda \in \Lambda}$ shows that $A$ is an AF-algebra.
We also observe that $A$ is realized as an increasing union
of the net $(\mathbb{M}_2(\mathbb{C}) \wr_G H_{\lambda})_{\lambda \in \Lambda}$.
Hence Proposition 2.1 in \cite{ORS} yields
the uniqueness of tracial states on $A$.
By Theorem \ref{Thm:nocentral}, $A$ has no non-trivial central sequences.

(2): Let $G$ be as in the proof of (1).
Put $M:=(\mathbb{M}_2(\mathbb{C}), \tau)\vwr G$.
It follows from Theorem \ref{Thm:nocentral}
that $M$ has no non-trivial central sequences.
In particular $M$ is a factor.
By the same reason as in (1),
$M$ possesses the desired properties.

(3): Again let $G$ be as in the proof of (1). Take a faithful non-tracial state $\varphi$ on $\mathbb{M}_2(\mathbb{C})$.
Then, by Exercise 13.4.2 in \cite{KR},
$(M, \varphi) \vwr G$ is of type III.
The other required conditions are confirmed by the same way as (2).

(4):
Choose a family $( \Gamma_\alpha)_{\alpha< \omega_1}$
of amenable groups
satisfying $\sharp \Gamma_0= \kappa$ and $\sharp \Gamma_\alpha = \aleph_0$ for $0<\alpha <\omega_1$.
Let $G$ be the iterated wreath product group of the family $(\Gamma_\alpha)_{\alpha< \omega_1}$.
Fix a Kirchberg algebra $B$.
Set $A:=B \wr G$.
Clearly $A$ is unital.
By Theorem 4.2.6 in \cite{BO}, $A$ is nuclear.
By Lemma \ref{Lem:pi}, $A$ is simple and purely infinite.
Since $\sharp G= \kappa$, the density character of $A$ is $\kappa$.
Again by Theorem \ref{Thm:nocentral}, $A$ has no non-trivial central sequences.
\end{proof}
\begin{Rem}\label{Rem:Cartan}
In contrast to the absence of regular abelian $\ast$-subalgebras in Theorem \ref{Thmint:prime}, one can arrange \Cs-algebras stated in Corollary \ref{Cor:NSKT}
to be a crossed product of a group acting on an abelian operator algebra.
Indeed, for (4), choose $B$ in the proof to be a Kirchberg algebra of the form
$C(X) \rc \Gamma$ for some amenable minimal topologically free dynamical system $\Gamma \curvearrowright X$
(see e.g., \cite{Suz13}).
Let $G$ denote the iterated wreath product group as in the proof of Corollary \ref{Cor:NSKT}.
We equip $X^G$ with the action of $\Gamma \wr G$ defined as follows.
The group $\bigoplus_G \Gamma$ acts on $ X^G$ coordinate-wise
and $G$ acts on $X^G$ by the left shifts.
We then obtain an isomorphism
\[A:=B \wr G \cong C(X^G) \rc (\Gamma \wr G).\]
We point out that the action $\Gamma \wr G \curvearrowright X^G$ is minimal
and topologically free. Thus
$C(X^G)$ in fact gives rise a {\it Cartan subalgebra} \cite{Ren08} of $A$.
A similar remark applies to operator algebras constructed in the proofs
of (1) to (3) of Theorem \ref{Thm:nocentral}.
(To see this, note that $\mathbb{M}_2(\mathbb{C}) \cong C(\mathbb{Z}_2) \rc \mathbb{Z}_2$ where
$\mathbb{Z}_2$ acts on itself by the left translations.)
\end{Rem}
\begin{Rem}
Recall that the cofinality of a limit ordinal $\alpha$ is
the cardinal \[{\rm cf}(\alpha):=\min\{\sharp S:S \subset \alpha: \sup(S) =\alpha\}.\]
By slightly modifying the proofs of Theorem \ref{Thm:nocentral} and Corollary \ref{Cor:NSKT} (replacing $\omega_1$ by $\alpha$),
one can show the following generalized statement:
for any limit ordinal $\alpha$,
one can construct operator algebras of density character $\sharp \alpha$ from each classes (1) to (4) of Corollary \ref{Cor:NSKT} with the following property.
There is no non-trivial central net $(x_i)_{i \in I}$ with $\sharp I< {\rm cf}(\alpha)$. (Here centrality and triviality of a bounded net is defined analogously to the sequence case.)
Note that for \Cs-algebras, this condition implies the non-existence of non-trivial tensor product components
of density character less than ${\rm cf}(\alpha)$.
We point out that when $\kappa$ is a successor (infinite) cardinal,
then ${\rm cf}(\kappa)$ is equal to $\kappa$.
In general, for any cardinals $\lambda < \kappa$,
one can find an ordinal $\alpha$ with
$\sharp\alpha=\kappa$, ${\rm cf}(\alpha)>\lambda$.
To see this, choose a successor cardinal $\mu$
with $\lambda < \mu \leq \kappa$.
Then the ordinal $\alpha:= \kappa + \mu$ possesses the desired property.
\end{Rem}
\section{Rigid sides of simple \Cs-algebras of decomposition rank one}\label{section:dr1}
We close this article by giving examples of simple unital monotracial \Cs-algebras which are tensorially prime and have decomposition rank one. This is another application of Popa's orthogonality method \cite{Pop83}.
\begin{Thm}\label{Thm:dr1}
For any cardinal $\kappa > \aleph_0$,
there is a tesorially prime simple unital monotracial \Cs-algebra of decomposition rank one, of density character $\kappa$, with no non-trivial central sequences.
\end{Thm}
\begin{proof}
Define $\Gamma_0 := \mathbb{Z}$, $\Gamma_1 := \bigoplus_\kappa \mathbb{Z}$,
and
$\Gamma_\alpha:=\mathbb{Z}$ for $2\leq \alpha<\omega_1$.
Let $G$ be the iterated wreath product group of the family $(\Gamma_\alpha)_{\alpha<\omega_1}$.
We show that $\mathbb{M}_2(\mathbb{C}) \wr G$ possesses the required properties.
We have already seen in the proof of Corollary \ref{Cor:NSKT} that $\mathbb{M}_2(\mathbb{C}) \wr G$ is simple, unital, monotracial, of density character $\kappa$, with no non-trivial central sequences.
To show that the decomposition rank of $\mathbb{M}_2(\mathbb{C}) \wr G$ is one,
as explained in Examples 6.1 (i) of \cite{KW}, it suffices to show that
$\mathbb{M}_2(\mathbb{C}) \wr G$ has nonzero K$_1$-group and has decomposition rank at most one.
To show the first claim, we define subgroups $H_\alpha \subset G_\alpha$, $\alpha \leq\omega_1$
as follows. Set $H_0:= \{e\}$, $H_1 := \{e \}$.
Assume we have defined $(H_\beta)_{\beta<\alpha}$ for some $1 < \alpha\leq\omega_1$.
We then define $H_\alpha$ as follows.
When $\alpha$ is a limit ordinal,
set $H_\alpha:= \bigcup_{\beta<\alpha} H_\beta$.
When $\alpha$ has the predecessor $\gamma$,
set $H_\alpha$ to be the subgroup of $G_\alpha$
generated by $\bigoplus_{G_\gamma}\Gamma_\gamma$ and $H_\gamma$.
By transfinite induction, we obtain the increasing family $(H_\alpha)_{\alpha\leq\omega_1}$ of subgroups in $G$.
We set $H:= H_{\omega_1}$.
It is not hard to show by transfinite induction that for each $1 \leq \alpha<\omega_1$, $H_\alpha$ is normal in $G_\alpha$,
$H_\alpha \cap G_1 = \{e\}$, and $H_\alpha \cup G_1$ generates $G_\alpha$.
This yields that $H$ is normal in $G$
and $G/H \cong G_1$.
Since $G_1$ is isomorphic to $\mathbb{Z}$,
we obtain the semidirect product decomposition
$G= H \rtimes G_1$.
This induces the crossed product decomposition
\[\mathbb{M}_2(\mathbb{C}) \wr G =A\rc G_1,\]
where $A:=\mathbb{M}_2(\mathbb{C}) \wrr{G} H$.
Since $A$ is unital and admits a tracial state,
the unit element $[1_A]_0$ is nonzero in $K_0(A)$.
The Pimsner--Voiculescu exact sequence \cite{PV}
then implies that $K_1(\mathbb{M}_2(\mathbb{C}) \wr G)$ is nonzero.
We next show the second claim.
By Theorem 2.1 of \cite{ORS} and Theorem 1.1 in \cite{Sat}, for each countable subgroups $L \subset K$ of $G$ with
$[K: L]=\infty$,
$\mathbb{M}_2(\mathbb{C}) \wr_K L$ is monotracial and $\mathcal{Z}$-stable.
By Theorem F of \cite{STW} and Theorem 3.8 of \cite{ORS}, each $\mathbb{M}_2(\mathbb{C}) \wr_K L$
has decomposition rank at most one.
This proves the second claim.

We next show the tensorial primeness of $\mathbb{M}_2(\mathbb{C}) \wr G$.
We identify $\mathbb{M}_2(\mathbb{C}) \wr G$
with a strongly dense \Cs-subalgebra of $(\mathbb{M}_2(\mathbb{C}), \tau) \vwr G$ in the canonical way.
To lead to a contradiction, assume we have a tensor product decomposition
\[\mathbb{M}_2(\mathbb{C}) \wr G = A_1 \otimes A_2\]
with both $A_i$ infinite dimensional.
By taking the strong closure, we obtain a tensor product decomposition 
\[(\mathbb{M}_2(\mathbb{C}), \tau) \vwr G =M_1 \votimes M_2,\]
where both $M_1, M_2$ are type II$_1$ factors.
We show that there is $\alpha< \omega_1$
satisfying $M_i \subset (\mathbb{M}_2(\mathbb{C}), \tau) \vwrr{G} G_\alpha$ for $i=1, 2$, which is obviously a contradiction.
To find such $\alpha$, choose a separable type II$_1$ subfactor
$N_i$ of $M_i$ for $i=1, 2$.
Take $\alpha<\omega_1$
satisfying $N_i \subset (\mathbb{M}_2(\mathbb{C}), \tau) \vwr G_\alpha$ for $i=1, 2$.
Let $U_1$, $U_2$ denote the unitary group
of the \Cs-subalgebras $\bigotimes _{G_\alpha} \mathbb{M}_2(\mathbb{C})$,
$\bigotimes _{G\setminus G_\alpha} \mathbb{M}_2(\mathbb{C})$ of $\bigvotimes _G(\mathbb{M}_2(\mathbb{C}), \tau)$ respectively.
For any $g\in G\setminus G_\alpha$, $v_1 \in U_1$, $v_2 \in U_2$, and $i=1, 2$,
we will show the relation
\[v_1 v_2 u_g N_i u_g^\ast v_2^\ast v_1^\ast \perp N_i.\]
Since the product $U_1 \cdot U_2$ spans a strongly dense subspace of $\bigvotimes _G(\mathbb{M}_2(\mathbb{C}), \tau)$,
this together with Corollary 2.6 of \cite{Pop83} concludes the desired inclusions
\[M_1 \subset \mathcal{N}_{(\mathbb{M}_2(\mathbb{C}), \tau) \vwr G}(N_2)''
\subset (\mathbb{M}_2(\mathbb{C}), \tau) \vwrr{G} G_\alpha,\]
\[M_2 \subset \mathcal{N}_{(\mathbb{M}_2(\mathbb{C}), \tau) \vwr G}(N_1)''
\subset (\mathbb{M}_2(\mathbb{C}), \tau) \vwrr{G} G_\alpha.\]
To prove the claim, as $U_1, N_1, N_2 \subset (\mathbb{M}_2(\mathbb{C}), \tau) \vwr G_\alpha$, it is enough to show
\[v_2 u_g [(\mathbb{M}_2(\mathbb{C}), \tau) \vwr G_\alpha] u_g^\ast v_2^\ast \perp [(\mathbb{M}_2(\mathbb{C}), \tau) \vwr G_\alpha]\]
for $g\in G \setminus G_\alpha$ and $v_2 \in U_2$.
To see this, it suffices to show the equation
\[\tau(v_2u_g xu_s u_g^\ast v_2^\ast y u_t) = \tau(v_2u_g x u_s u_g^\ast v_2^\ast)\tau(yu_t)\]
for $g\in G \setminus G_\alpha$, $v_2 \in U_2$, $x, y \in \bigotimes_{G_\alpha}\mathbb{M}_2(\mathbb{C})$, $s, t \in G_\alpha$.
When $\{s, t\} \neq \{e\}$,
by Lemma \ref{Lem:disjoint}, we have
\[\tau(v_2u_g xu_s u_g^\ast v_2^\ast y u_t) = 0=\tau(v_2u_g x u_s u_g^\ast v_2^\ast)\tau(yu_t).\]
We next consider the case $s= t=e$. In this case, observe that $v_2 u_g x u_g^\ast v_2^\ast \in \bigotimes_{G \setminus G_\alpha}\mathbb{M}_2(\mathbb{C})$,
while $y \in \bigotimes_{G_\alpha}\mathbb{M}_2(\mathbb{C})$.
We thus obtain
\[\tau(v_2u_g x u_g^\ast v_2^\ast y ) = \tau(v_2u_g x u_g^\ast v_2^\ast)\tau(y).\]
\end{proof}
\begin{Rem}
Let $G$ be the iterated wreath product of a family $(\Gamma_\alpha)_{\alpha<\omega_1}$,
where each $\Gamma_\alpha$ is countable infinite, amenable, and torsion-free.
Then, by the proof of Theorem \ref{Thm:dr1}, Proposition \ref{Prop:AFD}, and Remark \ref{Rem:Cartan}, $(\mathbb{M}_2(\mathbb{C}), \tau) \vwr G$ is
a prime AFD type II$_1$ factor with a regular maximal abelian $\ast$-subalgebra.
This construction also provides prime amenable type II$_1$ factors of
prescribed uncountable density character with a regular maximal abelian $\ast$-subalgebra. We do not know if these von Neumann algebras are AFD.
\end{Rem}
\subsection*{Acknowledgements}
Some parts of this work were carried out while the author was visiting Hokkaido University. The author is grateful to Reiji Tomatsu for his kind hospitality.
He also thanks Shuhei Masumoto for helpful comments on ordinals.
Finally, the author would like to thank the referees for their careful reading and valuable comments.
This work was supported by JSPS KAKENHI Grant-in-Aid for Young Scientists
(Start-up, No.~17H06737) and tenure track funds of Nagoya University.

\end{document}